\documentclass[UTF8]{amsart}
\usepackage{amsmath,amssymb,graphicx,bbm,amsthm,subfigure}
\usepackage[american]{babel}
\usepackage{color}
\usepackage[all]{xy}
\newtheorem{lemma}{Lemma}
\newtheorem{theorem}{Theorem}
\newtheorem{proposition}{Proposition}
\newtheorem{claim}{Claim}
\newtheorem{cor}{Corollary}

\newtheorem{remark}{Remark}
\numberwithin{equation}{section}

\DeclareMathOperator{\sys}{sys}
\DeclareMathOperator{\arccosh}{arccosh}
\usepackage{xcolor}
\usepackage{verbatim}

\usepackage{listings}
\usepackage{color} 
\definecolor{mygreen}{RGB}{28,172,0} 
\definecolor{mylilas}{RGB}{170,55,241}

\begin{document}

\lstset{language=Matlab,
    breaklines=true,
    morekeywords={matlab2tikz},
    keywordstyle=\color{blue},
    morekeywords=[2]{1}, keywordstyle=[2]{\color{black}},
    identifierstyle=\color{black},
    stringstyle=\color{mylilas},
    commentstyle=\color{mygreen},
    showstringspaces=false,
    numbers=left,
    numberstyle={\tiny \color{black}},
    numbersep=9pt, 
	xleftmargin=1cm
}

\title{Systoles of hyperbolic surfaces with big cyclic symmetry}

\author{Sheng Bai}
\address{School of Mathematical Sciences, Peking University, Beijing 100871, CHINA}
\email{barries@163.com}

\author{Yue Gao}
\address{School of Mathematical Sciences, Peking University, Beijing 100871, CHINA}
\email{yue\_gao@pku.edu.cn}

\author{Shicheng Wang}
\address{School of Mathematical Sciences, Peking University, Beijing 100871, CHINA}
\email{wangsc@math.pku.edu.cn}
\date{}

\maketitle

\begin{abstract}
		
We obtain the exact values of the systoles of these hyperbolic surfaces of genus $g$ with cyclic symmetries of the maximum order and the next maximum order.  
Precisely: for genus $g$ hyperbolic surface with order $4g+2$ cyclic symmetry, the systole is $2\arccosh (1+\cos \frac{\pi}{2g+1}+\cos \frac{2\pi}{2g+1})$ when $g\ge 7$, and  for genus $g$ hyperbolic surface with order $4g$ cyclic symmetry, the systole is $2\arccosh (1+2\cos \frac{\pi}{2g})$ when $g\ge4$.
\end{abstract}

\tableofcontents

\section{Introduction}
In this note all surfaces are closed and orientable,  all symmetries on surfaces
are  orientation preserving. When we talk about symmetries on hyperbolic surfaces, we assume those symmeties are isometries. 

  Hyperbolic surface is a fundamental research object in mathematics with a quite long history. 
  Systole is an important topic in this research.
    Systole on a closed hyperbolic surface indicates either a shortest closed geodesic or its  length, and we often used for latter.
 For a survey on the study of the systole, see Parlier \cite{parlier2014simple}.

Below we just list some results which close to our result.
 F. Jenni \cite{jenni1984ersten} got the maximal systole of genus $2$ surfaces and C. Bavard \cite{bavard1992systole} got that of genus $2$ and $5$ hyperelliptic surfaces. P. Schmutz \cite{schmutz1993reimann} obtained the systole of some surfaces constructed from convex uniform polyhedra with genus 3, 4, 5, 11, 23, 59. 

In a rather different way, P. Buser and P. C. Sarnak (\cite{buser1994period} constructed closed hyperbolic surfaces whose systole has a near-optimal asymptotic behavior with respect to the genus of the surface by arithmetic methods. Later Katz, Schaps and Vishne \cite{katz2007logarithmic} found a family of surface with Hurwitz symmetry and with systole not smaller than $4/3\log g$.
 See \cite{petri2015graphs} and  \cite{petri2018hyperbolic}  for more recent examples.

Our work is partly inspired  by the work \cite{katz2007logarithmic}:
 Classical results claims that if  a finite group $G$ acts on  $\Sigma_g$, then $|G|\le 84(g-1)$ (A. Hurwitz, \cite{Hu}), and moreover 
 $|G|\le 4g+2$ if $G$ is cyclic (A. Wiman \cite{wiman1895ueber}). Call a topological/hyperbolic surface $\Sigma_g$ has Hurwitz symmetry,
 if it admits a finite group action of order $84(g-1)$ and has Wiman symmetry if it admits a cyclic finite group action of order $4g+2$.
 Note it is known for infinitely many $g$, topological surface $\Sigma_g$ has Hurwitz symmetry, and for every $g>1$, topological surface 
 $\Sigma_g$ has Wiman symmetry. Moreover it is also well known that the next biggest cyclic symmetry on topological surface $\Sigma_g$ has order $4g$ \cite{kulkarni1997riemann}.
 It is a classical result that each periodical map $f$ on $\Sigma_g$ can be realized as an isometry for some hyperbolic structure $\rho$ on $\Sigma_g$.
 
Our result is
	
\begin{theorem}\label{Main1} Suppose $\Sigma_g^1$ and $\Sigma_g^2$ are hyperbolic surfaces with cyclic symmetry of order $4g$
and $4g+2$ respectively. Then
			
		$$		\sys(\Sigma_g^1)=2\arccosh (1+2\cos \frac{\pi}{2g})\, \, \text{ for} \,\, {g\ge 4}$$

		$$		\sys(\Sigma_g^2) = 2\arccosh (1+\cos \frac{\pi}{2g+1}+\cos \frac{2\pi}{2g+1})\, \, \text{ for} \,\, {g\ge 7}$$
	
\end{theorem}

We list some values of the systoles, see Table \ref{tab_value}. 
\begin{table}
		\centering
		\begin{tabular}{llll}
				Genus & $\sys(\Sigma^1_g)$ &Genus & $\sys(\Sigma^2_g)$ \\
				4 & 3.41464123 & 7 & 3.44730852 \\
				5 & 3.45497357 & 8 & 3.46473555 \\
				6 & 3.47667914 & 9 & 3.47691634 \\
				7 & 3.48969921 & 10 & 3.48576585
		\end{tabular}
		\caption{}
		\label{tab_value}
\end{table}

From the  proof of the theorem, we have 

\begin{cor}
In $\Sigma_g^1$, there are $2g$ closed geodesics having length  $\sys(\Sigma_g^1)$
and in $\Sigma_g^2$, there are $2g+1$ closed geodesics having length  $\sys(\Sigma_g^2)$.
\end{cor}

The method to prove above result is rather direct, and seems  different from previous work we mentioned.

We first verify the hyperbolic structure of $\Sigma_g$ with isometries of big  order is unique
up to homeomorphism of $\Sigma_g$. Then we just pick a well known model of hyperbolic surface $\Sigma_{g}$ with  isometries of order $4g$ (respectively $4g+2$). In this $\Sigma_g$, we conjecture a simple closed geodesic $\gamma$ realizing the systole.
We calculate the length $2h$ of $\gamma$. Then we devote to the proof that the injective radius of $\Sigma_g$ is $h$. 

In Section 2 we  verify the uniqueness of hyperbolic structure of $\Sigma_g$ with isometries of big  order. 
In section 3, we give some trigonometric formulae which will be used later. 
Some formulae are copied from \cite{buser2010geometry} and some are derived by us. 
Theorem \ref{Main1} is proved in Section 4 and Section 5.

The paper is self-contained up to several standard text book.

{\bf Acknowledgement: } We thank Professor Ursula Hamenst\"adt for helpful communication. 
The authors are supported by grant
No.11711021 of the National Natural Science Foundation of China.

\section{Uniqueness of hyperbolic $\Sigma_g$ which admits big cyclic symmetry}

Let $S^2(p,q,r)$ denote the hyperbolic orbifold with base space $S^2$ and 3 singular points of index $p,q, r$ respectively,
where $p, q,r$ are three positive integers. 

\begin{proposition}\label{unique1} Let $G_i$ be a cyclic group action on hyperbolic surface $(\Sigma_g, \rho_i)$, $i=1,2,$ such that 

(1) those two actions are conjugated, 

(2) $\Sigma _g/G_i= S^2(p,q,r)$.

Then  those two hyperbolic metric $\rho_1$ and $\rho_2$ on $\Sigma_g$ are isometric.
		
\end{proposition}

\begin{proof} For short, we write $\Sigma_g^i=(\Sigma_g, \rho_i)$. By (1) we have a homeomorphism $\phi:\Sigma_g^1
\to \Sigma_g^2$ such that $G_2=\phi G_1 \phi^{-1}$. By choosing the suitable generators $\sigma_1, \sigma_2$ of $G_1$ and $G_2$,
we may have $G_1=<\sigma_1>,$ $G_2=<\sigma_2>$, and $\sigma_2= \phi \sigma_1 \phi^{-1}$.
		$\sigma_1,\sigma_2$ are isometries on $\Sigma_g^1$ and $\Sigma_g^2$ respectively.  There is a commute diagram:

		\begin{equation}
				\xymatrix{
						\Sigma_g^1 \ar[d]^{\pi_1} \ar[r]^{\phi} & \Sigma_g^2 \ar[d]^{\pi_2} \\
						\Sigma_g^1/\sigma_1 \ar[r]^{\bar \phi} & \Sigma_g^2/\sigma_2
				}
				\label{for_diagram}
		\end{equation}

		$\pi_1$ and $\pi_2$ are the branch covers induced by $\sigma_1,\sigma_2$. $\phi$ and $\bar \phi$ are homeomorphisms between $\Sigma_g^1$ and $\Sigma_g^2$, $\Sigma_g^1/\sigma_1$ and $\Sigma_g^2/\sigma_2$ respectively. Both $\Sigma_g^1/\sigma_1$ and $\Sigma_g^2/\sigma_2$ are the hyperbolic orbifold $S^2(p,q,r)$. 

		What we need to prove is that there are isometries $\psi,\bar \psi$ between $\Sigma_g^1$ and $\Sigma_g^2$, $\Sigma_g^1/\sigma_1$ and $\Sigma_g^2/\sigma_2$ respectively, satisfying the commmute diagram (\ref{for_diagram_iso}). 

		The hyperbolic structure of the orbifold $S^2(p,q,r)$ is unique, Since it is obtained by doubling two hyperbolic
		triangles 
		$\Delta(p,q,r)$ of angles $\pi/p, \pi/q, \pi/r$, and the $\Delta(p,q,r)$ is unique.

		For a homeomorphism $\bar \phi$ on hyperbolic orbifold $S^2(p,q,r)$ that satisfies the diagram (\ref{for_diagram}), we are going to prove that $\bar \phi$ is isotopic to an isometry $\bar \psi$. 

		\begin{equation}
				\xymatrix{
						\Sigma_g^1 \ar[d]^{\pi_1} \ar[r]^{\psi} & \Sigma_g^2 \ar[d]^{\pi_2} \\
						\Sigma_g^1/\sigma_1 \ar[r]^{\bar \psi} & \Sigma_g^2/\sigma_2
				}
				\label{for_diagram_iso}
		\end{equation}

		For convenience, we just denote the singular points of the orbifold $S^2(p,q,r)$ by $p,q,r$ respectively. 

		We assume $l$ is the shortest geodesic connecting $p$ and $q$. Then $\bar \phi(l)$ is a curve connecting $\bar \phi(p)$ and $\bar \phi(q)$. By the definition of homeomorphism between orbifolds, $\bar \phi(p)$ and $p$ are singular points on the orbifold with the same order and $\bar \phi(q)$ and $q$ are singular points on the orbifold with the same order too. Without loss of generality, we assume that $\bar \phi(p)=p$ and $\bar \phi(q)=q$. 

		$|S^2(p,q,r)\backslash \{r\}|$ is an open  disk. By the contractability of disks, $l$ is isotopic to $\bar \phi(l)$ by an isotopy on $S^2(p,q,r)\backslash \{r\}$. So we may assume that $\bar \phi$ is already the identity on $l$.

		Similarly, consider the shortest geodesic $m$ between $p$ and $r$, $\bar\phi(m)$ is isotopic to $m$ in $S^2(p,q,r)\backslash l$. 

		Futhermore, the shortest geodesic $n$ between $q$ and $r$, $\bar\phi(n)$ is isotopic to $n$ in $S^2(p,q,r)\backslash (l\cup m)$. 

		Then there is a homeomorphism $\bar \psi$ on the orbifold isotopic to $\bar\phi$ and $\bar \psi |_{l\cup m \cup n}=id$. 

		Then since	$S^2(p,q,r)\backslash (l\cup m \cup n)$  consists of two congruent triangles, clearly the restriction of $\bar \psi$ on each triangle is isotopic to the identity, and we get that $\bar \phi$ is isotopic to the isometry $\bar \psi$. 
		
		$\bar \psi$ can be lifted to a homeomorphism ${\psi}$ between $\Sigma_g^1$ and $\Sigma_g^2$ that satisfies the diagram (\ref{for_diagram_iso}) and is isotopic to ${\phi}$. Since $\bar\psi\circ\pi_1$ and $\pi_2$ are local isometry, ${\psi}$ is a local isometry. Since ${\psi}$ is bijective, ${\psi}$ is an isometry.
\end{proof}

\begin{cor}\label{unique2}
The hyperbolic structure of surface $\Sigma_g$ with cyclic symmetry of order $k$ is unique for $k=4g+2$ and $4g$. 
\end{cor}

\begin{proof} 
If $G$ is a cyclic group action of order $4g$ on the hyperbolic suraces $\Sigma_g$, then the orbifold $\Sigma_g/G$ is $S^2(4g, 4g, 2)$.

If $G$ is a cyclic group action of order $4g+2$ on the hyperbolic suraces $\Sigma_g$, then the orbifold $\Sigma_g/G$ is $S^2(4g+2, 2g+1, 2)$.

It is well-known that any two cyclic group action of order $k$ on $\Sigma_g$ is unique up to conjugacy, see  \cite{kulkarni1997riemann}, or quick argument in \cite{guo2015embedding}, where $k=4g$ or $4g+2$.

By Proposition \ref{unique1}, the corollary follows.
\end{proof}

\section{Trigonometric formulae}

Below are the trigonometric formulae we use in this work: 

A. $\cosh 2x=2\cosh^2 x-1 = 2\sinh^2 x+1$

Since 
\begin{eqnarray*}
		(e^x+e^{-x})^2 &=& e^{2x}+e^{-2x}+2 \\
		(2\cosh x)^2 &=& 2\cosh 2x + 2,
\end{eqnarray*}
 therefore 
\begin{equation}
\cosh 2x=2\cosh^2 x-1 = 2\sinh^2 x+1. 
		\label{for_2x}
\end{equation}

B. {Right angle triangle (Figure \ref{fig_right_angle_triangle}) \cite[p.454]{buser2010geometry}} 
\begin{eqnarray}
		\cosh c &=& \cosh a \cosh b \label{for_tri_1}.\\
		\sinh a &=& \sin \alpha \sinh c \label{for_tri_3}.\\
		\cos \alpha &=& \cosh a \sin \beta \label{for_tri_5}. 
\end{eqnarray}

C. {Cosine law and sine law of hyperbolic triangles (Figure \ref{fig_triangle_2}) \cite[p.454]{buser2010geometry}}

\begin{figure}[htbp]
\centering
\begin{minipage}[t]{0.3\textwidth}
\centering
\includegraphics{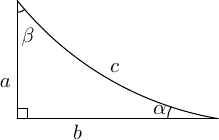}
\caption{}
\label{fig_right_angle_triangle}
\end{minipage}
\begin{minipage}[t]{0.3\textwidth}
\centering
\includegraphics{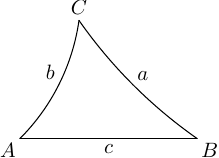}
\caption{}
\label{fig_triangle_2}
\end{minipage}

\begin{minipage}[t]{0.3\textwidth}
\centering
\includegraphics{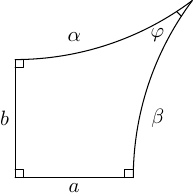}
\caption{}
\label{fig_trirectangle}
\end{minipage}
\begin{minipage}[t]{0.3\textwidth}
\centering
\includegraphics{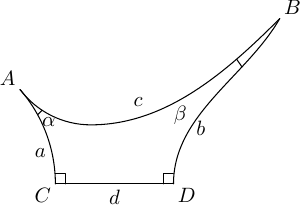}
\caption{}
\label{fig_2_right_angle_quadrilateral}
\end{minipage}

\end{figure}

\begin{eqnarray}
		\cosh c &=& -\sinh a \sinh b \cos C + \cosh a \cosh b\label{for_tri_cos}. \\
		\frac{\sinh a}{ \sin A} &=& \frac{\sinh b}{ \sin B}\label{for_tri_sin}. 
\end{eqnarray}

D. Trirectangle (Figure \ref{fig_trirectangle}) \cite[p.454]{buser2010geometry}
\begin{eqnarray}
		\cosh a &=& \cosh \alpha \sin \varphi \label{for_trirect_3}. \\
		\sinh \alpha &=& \sinh a \cosh \beta \label{for_trirect_5}.\\
		\sinh \alpha &=& \coth b\cot \varphi. \label{for_trirect_6}
\end{eqnarray}

E. {Quadrilateral with two right angles (Figure \ref{fig_2_right_angle_quadrilateral})}

\begin{eqnarray}
		\cosh c &=& \cosh d \cosh a \cosh b-\sinh a \sinh b \label{for_birect_1}.\\
		\cosh d &=& \sin \alpha \sin \beta \cosh c - \cos \alpha\cos\beta \label{for_birect_2}. 
\end{eqnarray}

Below are the proofs for the last  
two formulae:

\begin{proof}[Proof of (\ref{for_birect_1})] In Figure \ref{fig_2_right_angle_quadrilateral_2}, 
in $\triangle BCD$, we have $\cosh e = \cosh b \cosh d$ by (\ref{for_tri_1}), and $\sin\angle BCD = \sinh b/\sinh e$ by  (\ref{for_tri_3}).

\begin{figure}[htbp]
		\centering
		\includegraphics{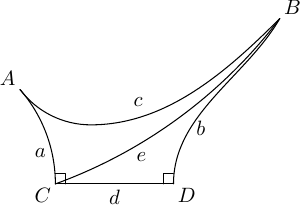}
		\caption{}
		\label{fig_2_right_angle_quadrilateral_2}
\end{figure}

Then in $\triangle ACB$, 
$\cosh c= -\sinh a \sinh e \cos \angle ACB + \cosh a \cosh e $ by  (\ref{for_tri_cos}). 
Since  $\angle BCD+\angle ACB=\pi/2$, we have  $\cos \angle ACB=\sin \angle BCD$. By (\ref{for_tri_3}) we have 
$\sin \angle BCD=\frac{\sinh b}{\sinh e}$. Plug  the second and third formula into the first formula, we obtain (\ref{for_birect_1}). \end{proof}

\begin{proof}[Proof of (\ref{for_birect_2})]
We put the quadrilateral in upper half plane. For points in the real line, we will talk their coordinates and additions as real numbers.

The quadrilateral $ACDB$ in Figure \ref{fig_2_right_angle_quadrilateral} is corresponding to the quadrilateral $T_1R_1R_2T_2$ in Figure \ref{fig_2_right_angle_quadrilateral_3}. In this quadrilateral, $\angle T_1=\alpha$, $\angle T_2=\beta$, $\angle R_1 = \angle R_2 = \pi/2$. The length of $T_1T_2=c$.  We hope to obtain the length of $R_1R_2$, (denoted $d$).
 
We assume the coordinate of $T_1$ to be $i$, $T_2$ to be $i e^c$. 
In Figure \ref{fig_2_right_angle_quadrilateral_3}, let $A, B$ be the centers of small and big half circles respectively, and $r, R$ be the  radius of small and big half circles respectively. By Euclidean trigonometry, in $\triangle AT_1O$, $\angle O =\pi/2$, $\angle A = \alpha$. In $\triangle BT_2O$, $OT_2=e^c$ and $\angle B=\beta$, so we have 
\begin{equation}
		A=\cot \alpha,\,\,  r=1/\sin \alpha, \,\,  B=-e^c \cot \beta\,\, R=e^c/\sin\beta, \label{equ_bisect}
\end{equation}

Respectively the coordinate of $C$, $D$, $E$, $F$ are $A-r$, $A+r$, $B-R$, $B+R$.

\begin{figure}[htbp]
		\centering
		\includegraphics{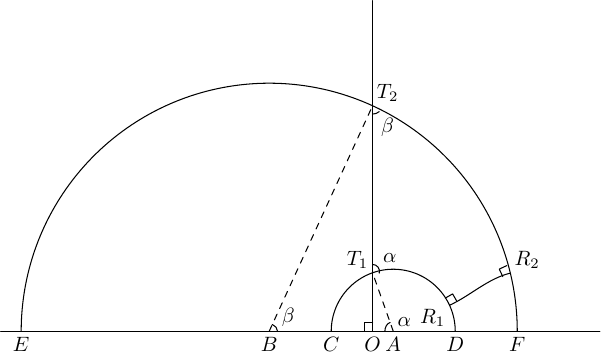}
		\caption{}
		\label{fig_2_right_angle_quadrilateral_3}
\end{figure}

To get $|R_1R_2|$ we need the isometry $\varphi$ of $\mathbb{H}^2$ below which sends $B-R$ to $\infty$
, where
\[
		\varphi:z\mapsto -1/(z-(B-R)). 
\]

Figure \ref{fig_2_right_angle_quadrilateral_4} is the image of Figure \ref{fig_2_right_angle_quadrilateral_3} under $\varphi$,
and  $E', F', C', D'$ are the images of $E,F, C, D$ under $\varphi$ respectively.
Their coordinates in Figure \ref{fig_2_right_angle_quadrilateral_4} are list below by the coordinates of their preimages in Figure \ref{fig_2_right_angle_quadrilateral_3}. 
\begin{equation}
		E' =\infty,\,\,
		F' = -\frac{1}{2R}, \,\,
		C' = \frac{-1}{A-r-(B-R)}, \,\,
		D' =  \frac{-1}{A+r-(B-R)}.  \label{equ_bisect_prime}
\end{equation}
We denote the center of the halfcircle as $P$. 

We figure out the formula of the length $|R_1'R_2'|  $:
\begin{equation}
		|R_1'R_2'| = \arccosh \frac{1}{\cos \angle R_1'F'R_2' }. 
		\label{for_birect_2_r1r2}
\end{equation}
 
In this proof, to avoid confusions, we denote the hyperbolic distance between two points $A$ and $B$ as $|AB|$, 
the Euclidean distance between them as $|AB|_e$. 

Since $\angle R_2'F'P = \angle PR_1'F' = \pi/2$, then $\angle R_1'F'R_2' = \angle R_1'P F'$. Thus 
\[
	\cos	\angle R_1'F'R_2' = \cos \angle R_1'P F' = \frac{|R_1'P|_e}{|PF'|_e}.
\]

The Euclidean distance $|R_1'P|_e$, $|PF'|_e$ is obtained directly by the coordinates of the points. Here $|PR_1'|_e = |PD'|_e = |C'D'|_e/2 $.
Therefore
\begin{eqnarray*}
		\cosh |R_1R_2| &=& \frac{1}{\cos \angle R_1'F'R_2'} 
		= \frac{1}{\cos \angle R_1'PF'} 
		= \frac{|F'P|_e}{|PR_1'|_e} \\
		&=&  \frac{-\frac{1}{2R}+(\frac{1}{A-r-(B-R)}+\frac{1}{A+r-(B-R)})/2}{(\frac{1}{A-r-(B-R)}-\frac{1}{A+r-(B-R)})/2}\,\,\,\, \,\,( \text{by pluging (\ref{equ_bisect_prime})})\\
		&=& \frac{2(A-B+R) - \frac{(A-B+R)^2-r^2}{R}}{2r} \\
		&=& \frac{2(A-B+R)R-(A-B+R)^2+r^2}{2Rr} \\
		&=& \frac{R^2-(A-B)^2+r^2}{2Rr} \\
		&=& \frac{e^{2c}\sin^2 \alpha - (\cos \alpha \sin \beta + e^c \sin \alpha \cos \beta)^2 + \sin^2 \beta}{2e^c \sin\alpha \sin \beta} \,\,\,\,\,\,(\text{by  pluging (\ref{equ_bisect})})\\
		&=& \frac{e^{2c}\sin^2 \alpha + \sin^2 \beta - (\cos^2 \alpha \sin^2 \beta + 2e^c \sin \alpha \cos \alpha \sin \beta \cos \beta + e^{2c} \cos^2 \beta \sin^2 \alpha) }{2e^c \sin\alpha \sin \beta} \\
		&=& \frac{(e^{2c}+1)\sin^2 \alpha\sin^2\beta - 2e^c \sin\alpha \cos \alpha \sin \beta \cos \beta }{2e^c \sin \alpha \sin \beta} \\
		&=& \cosh c \sin \alpha \sin \beta - \cos \alpha \cos \beta. 
\end{eqnarray*}

\begin{figure}[htbp]
		\centering
		\includegraphics{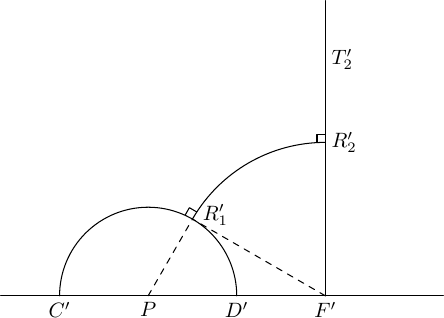}
		\caption{}
		
		\label{fig_2_right_angle_quadrilateral_4}
\end{figure}
\end{proof}

\section{Polygon modules of hyperbolic surfaces and candidates of systoles}

By Corollary \ref{unique2}, to prove Theorem \ref{Main1}, we need only to work on a concrete module of hyperbolic surface $\Sigma_g$ of given symmetry.

\subsection{Polygon modules}

(1) $4g$-cyclic symmmtry case:
The hyperbolic surface of genus $g$ with $4g$-cyclic symmetry can be obtained  by identifying  opposite edges the regular hyperbolic 4g-polygon of angle sum $2\pi$.  Below we denote this surface by $\Sigma^1_{g}$ and the polygon by $P_{1,g}$, and often $P_1$ for short.
Note all vertices in $P_{1}$ are identified to  one point in $\Sigma^1_{g}$. Each angle of $P_{1,g}$ is $2\pi/(4g)$.
We often view $P_{1}$ as a polygon contained either in $\Sigma^1_{g}$ or in hyperbolic plane $\mathbb{H}^2$. 

\begin{figure}[htbp]
\centering
\begin{minipage}[t]{0.45\textwidth}
\centering
\includegraphics{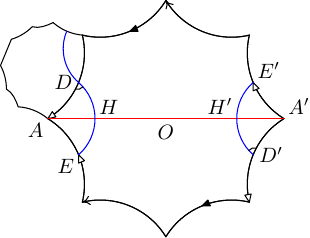}
\caption{}
\label{fig_4g_gon_4}
\end{minipage}
\end{figure}

For case of $4g+2$-cyclic symmetry we will pick two module.

(2) First polygon module of $4g+2$-cyclic case:

The hyperbolic surface of genus $g$ with $4g+2$-cyclic symmetry can be obtained  by identifying  opposite edges the regular hyperbolic 4g+2-polygon of angle sum $4\pi$.  Below we denote this surface by $\Sigma^2_{g}$ and the polygon by $P_{2,g}$, often $P_2$ for short. 
Note all vertices in $P_{2}$ are alternatively  identified to two points in $\Sigma^2_{g}$. Each angle of $P_{2}$ is $2\pi/(2g+1)$.
We often view $P_{2}$ as a polygon contained either in $\Sigma^2_{g}$ or in the hyperbolic plane $\mathbb{H}^2$ .

\begin{figure}[htbp]
\centering
\begin{minipage}[t]{0.45\textwidth}

		\centering
		\includegraphics{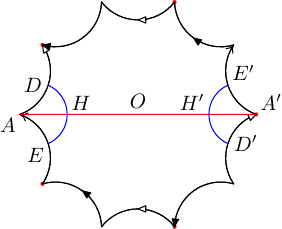}
		\caption{}
		\label{fig_4g_plus_2_gon_4}
\end{minipage}
\begin{minipage}[t]{0.45\textwidth}
		\centering
		\includegraphics{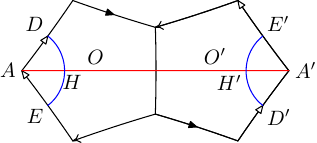}
		\caption{}
		\label{fig_4g_plus_2_gon_5_4}
\end{minipage}
\end{figure}

(3) Dual fundamental region, the second polygon module of $4g+2$ case. 

For regular hyperbolic $n$-gon $P_i\subset \mathbb{H}^2$, $i=1,2$, defined above, 
$D_i(P_i)$  give a tessellation of $\mathbb{H}^2$, where $D_i$ is the deck transformation group such that $\mathbb{H}^2/D_i=\Sigma^i_g$.
If two polygons in  $D_i(P_i)$ share one edge, connect their centers by the unique geodesic arc. The union of those arcs form a lattice,
which provide a new tesselation $D_i(P_i^*)$, where $P_i^*$ is also a fundamental region of $\Sigma_g^i$.

\begin{figure}[htbp]
		\centering
\begin{minipage}[t]{0.45\textwidth}
		\centering
		\includegraphics{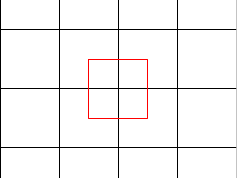}
		\caption{}
		\label{fig_dual_1}
\end{minipage}
\begin{minipage}[t]{0.45\textwidth}
		\centering
		\includegraphics{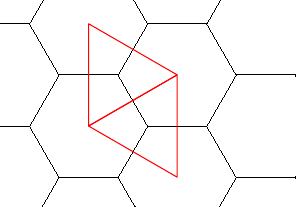}
		\caption{}
		\label{fig_dual_2}
\end{minipage}
\end{figure}

One can check directly that 

\begin{proposition}
(1) $P_1^*$ is congruent to $P_1$, that is $P_1^*$ is also a regular hyperbolic 
4g-polygon of angle sum $2\pi$. For each vertex $A$ of $P_1$, there is a unique $P_1^*$ centered at $A$.

(2) While $P_2^*$ is a union of two regular hyperbolic $2g+1$ polygons gluing along a pair of edges (see Figure \ref{fig_4g_plus_2_gon_5_4}),
and each angle of the regular hyperbolic $2g+1$ polygon is $\pi/(2g+1)$. For each vertex $A$ of $P_2$, there is a unique $P_2^*$ such that one of its regular $(2g+1)$-gon centered at $A$.
\end{proposition}

\begin{remark}  (1) Figure \ref{fig_dual_1} and Figure \ref{fig_dual_2} are the genus one counterparts of $P_1 (P_1^*)$ and  $P_2 (P_2^*)$ respectively.

(2) $\Sigma^2_{g}$ is also obtained  by identifying  opposite edges 
the hyperbolic 4g-polygon $P^*_{2}$.
Note all vertices in $P^*_{2}$ are  identifed to one point in $\Sigma^2_{g}$.
To see the $4g+2$-cyclic symmetry on $P_2^*$ directly, see \cite{wang1991maximum}.
\end{remark}

\subsection{Candidates of systoles}

\label{subsect_cal}
In the regular polygons $P_i$, $i=1,2$,  we call the geodesic that connects two opposite vertices a \emph{diameter}. 
In $P_2^*$, call the geodesic connects two opposite vertices which is perpendicular to the common edge of two $2g+1$-regular polygons
a diameter.

Now we will use one polygon $P$ below to present $P_1$, $P_2$ and $P_2^*$.
where $O$ is the center of the whole $P$ when $P$ is either $P_1$ or $P_2$,  and is the center of one of two $2g+1$ polygon when 
$P=P_2^*$, $AA'$ is a chosen diameter of the polygon $P$. 
Let $D, E, D', E'$ be the mid-points of the corresponding edges neighboring the diameter.
Then by symmetry it is not difficult to observe that the geodesic segments $DE$ and  $E'D'$ form a closed geodesic $\gamma_i$ in $\Sigma_g^i$. 

We are going to calculate the length of $\gamma_i$, $i=1,2$.
Let $a$ be the angle between $OE$ and $OA$, and $b$ be the half of the angle at the vertex of the polygon.
Then  (see Figure \ref{fig_4g_plus_2_gon_10}). 
\begin{equation}
		a=b=\frac{\pi}{4g} \,\text{for}\, {P_1}; \,\, 2a=b=\frac{2\pi}{4g+2} \, \text{for}\,  P_2; \,\, a=2b=\frac{2\pi}{4g+2}, \text{for}\,  P_2^*. \label{for_ab}
\end{equation}

\begin{figure}[htbp]
		\centering
		\includegraphics{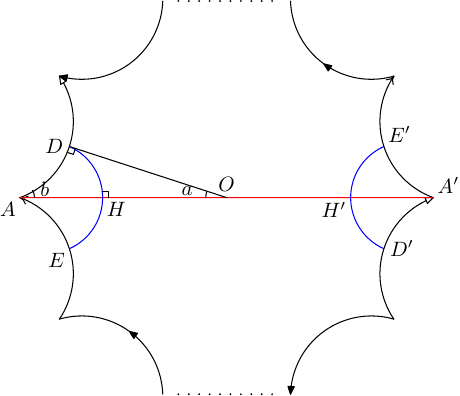}
		\caption{}
		\label{fig_4g_plus_2_gon_10}
\end{figure}

\begin{proposition}\label{length}
		In the polygons (see figure \ref{fig_4g_plus_2_gon_10}), 
		
		(1) $\cosh |OD| = \cos b /\sin a $, 
	 
		(2) $\cosh |AD| = \cos a / \sin b$, 
		
		(3) $\cosh |OA| = \cot a \cot b$,  

(4) $\sinh |DH| = \sqrt{ \cos^2 a - \sin^2 b }$,

(5) $\cosh |DE| =2\left( \cos^2 a - \sin^2 b \right) +1$,   

(6) $\cosh |AH| =\frac{ \cos a}{\sin b \sqrt{\cos^2 a+ \cos^2 b}}$, 

(7) $\sinh |AH|= \cot b \sqrt{\frac{\cos^2 a - \sin^2 b}{\cos^2 a + \cos^2 b}}$.

		In the surfaces the length of the geodesic $\gamma_i$ ($DEE'D'$) is $2\arccosh (1+\cos 2a+\cos 2b)$. 

		\label{prop_length_10}
\end{proposition}

\begin{proof}
		In Figure \ref{fig_4g_plus_2_gon_10}, in the right-angle triangle $\triangle OAD$, $\angle ODA = \pi/2$, $\angle AOD = a$ and $\angle ODA = b$. 
		Then by  (\ref{for_tri_5}), \[\cosh |OD| = \frac{\cos \angle OAD}{\sin \angle AOD} = \frac{\cos b}{\sin a},\] and \[\cosh |AD| = \frac{\cos \angle AOD}{\sin \angle DAO} = \frac{\cos a}{\sin b }.\] By  (\ref{for_tri_1}), \[\cosh |OA| = \cosh |OD| \cosh |AD| = \cot a \cot b.\] 
		In the right-angled triangle $\triangle ADH$, $\angle DHA = \pi/2$, $\angle DAH = b$ and $\cosh |AD| = \cos a/\sin b $. Then by  (\ref{for_tri_3}), 
		\begin{eqnarray*}
			\sinh |DH| &=& \sinh |AD| \sin \angle DAH 	\\
			&=& \sqrt{ \left( \frac{\cos a}{\sin b } \right)^2-1 } \sin b\\
			&=& \sqrt{ \cos^2 a- \sin^2 b }.
		\end{eqnarray*} 
Therefore, 
\begin{eqnarray*}
		\cosh |DE| &=& 2\sinh^2 |DH| +1 \\
		&=& 2\left( \cos^2 a - \sin^2 b \right) +1 \\
		&=& \cos 2a + \cos 2b + 1. 
\end{eqnarray*}

		Then 
		$\cosh |DH| = \sqrt{\sinh^2 |DH| +1} = \sqrt{\cos^2 a + \cos^2 b}$ and by  (\ref{for_tri_1}), 
		\begin{eqnarray*}
				\cosh |AH| &=& \cosh|AD|/\cosh |DH|\\ &=&\frac{ \cos a}{\sin b \sqrt{\cos^2 a + \cos^2 b} } .
		\end{eqnarray*} 

		\begin{eqnarray*}
				\sinh |AH| &=& \sqrt{\cosh^2 |DH| -1} \\
				&=& 
				\cot b \sqrt{\frac{\cos^2 a - \sin^2 b}{\cos^2 a + \cos^2 b}}.
		\end{eqnarray*}

		The length of the geodesic $\gamma_i$ ($DEE'D'$) is equal to $2|DE|$. Thus by (5) we have
		\begin{equation}
				|\gamma_i| = 2\arccosh (1+\cos 2a+\cos 2b). \label{for_gamma_i}
		\end{equation}
\end{proof}

Let $2h_i$ be the length of $\gamma_i$,  by plugging (\ref{for_ab}) into (\ref{for_gamma_i}) we get  

\begin{cor}
$$h_1=\arccosh (\cos \pi/2g+1)\,  \text{and}\,  h_2=\arccosh (1+\cos \frac{\pi}{2g+1}+\cos \frac{2\pi}{2g+1})$$
\end{cor}

\section{Injective radious $h_i$}

\subsection{Some reductions}

Now we begin to prove $2h_i$ is the systole of the hyperbolic surface $\Sigma^i_g$, $i=1,2$. 
By definitions and Proposition \ref{length}, we have the following 

\begin{claim}
$B(x,h_i)\to \Sigma^i_g$ is an embedding  for each $x\in \Sigma_g^i$ implies that $\sys(\Sigma^i_g)=2h_i$, $i=1,2$.  
		\label{claim_origin}
\end{claim}

\begin{figure}[htbp]
\centering
\begin{minipage}[t]{0.45\textwidth}
\centering
\includegraphics{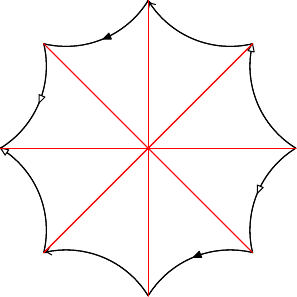}
\caption{}
\label{fig_4g_gon_9}
\end{minipage}
\end{figure}

We view $P_i$ as a polygon contained in $\Sigma^i_g$. 
There is an observation that every closed essential geodesic in $\Sigma_g^i$ 
must meet a diameter of $P_i$, since the diameters of $P_i$ cut $\Sigma_g^i$ into contractible pieces (see Figure \ref{fig_4g_gon_9}).
In particular each closed shortest geodesic meets a diameter of $P_i$.
Note that  all points  in closed  shortest geodesics have the same injective radious (see page \cite[p.178]{petersen2006riemannian}). 

Therefore we can reduce Claim \ref{claim_origin} to the following: 

\begin{claim}
$B(x,h_i)\to \Sigma^i_g$ is an embedding	 for each $x\in AA'$ implies that $\sys(\Sigma^i_g)=2h_i$, , $i=1,2$.  
		
		\label{claim_med}	
\end{claim}

Let $H$, $H'$ be the intersections of the diameter $AA'$ and the geodesic segment $DE$ and $E'D'$ respectively, see Figure \ref{fig_4g_plus_2_gon_10}.

In each of the following two pictures Figure \ref{fig_4g_gon_5} and \ref{fig_4g_plus_2_gon_5}, 
the black polygon is  $P_i$, and the green one is the dual polygon  $P_i^*$ around the vertex $A$, where Figure \ref{fig_4g_gon_5} is for $i=1$ and Figure \ref{fig_4g_plus_2_gon_5} is for $i=2$. By symmetry and Claim \ref{claim_med}, we need only consider $x\in OA\subset P_i$. 
Note $OA=OH\cup HA$  and $A$ is the center of $P_1^*$ in Figure \ref{fig_4g_gon_5} and the center of one of the two $2g+1$ polygon of $P_2^*$.
Moreover $P^*_1$ is a polygon congruent to $P_1$.  
We reduce Claim \ref{claim_med} to the following 

\begin{claim}
(1) $B(x,h_1)\to \Sigma_g^1$ is an embedding for each $ x\in OH\subset P_1$ implies that $\sys(\Sigma^1_g)=2h_1$; 

(2) $B(x,h_2)\to \Sigma_g^2$ is an embedding for each $ x\in OH\subset P_2$ and for each $x\in OH\subset P_2^*$
implies that $\sys(\Sigma^i_g)=2h_i$    
		\label{claim_final}	
\end{claim}
We remind the reader that in $P_2^*$, $O$ is the center of one of two $2g+1$ polygons as is defined in Section \ref{subsect_cal}. 

\begin{figure}[htbp]
		\centering
		\includegraphics{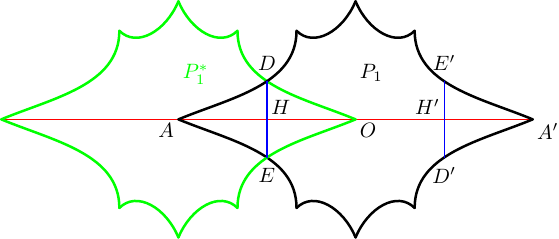}
		\caption{}
		\label{fig_4g_gon_5}
\end{figure}

\begin{figure}[htbp]
		\centering
		\includegraphics{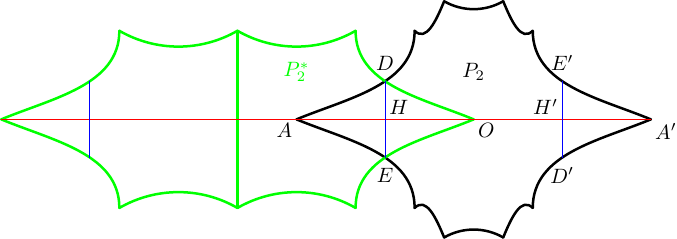}
		\caption{}
		\label{fig_4g_plus_2_gon_5}
\end{figure}

\subsection{Lift $B(x, h_i)$ to the universal cover}

Below for  $B(x, h_i)$ we assume that $x\in OH\subset P$.

\begin{proposition} For the polygon $P$,

		(1) The distance between any vertex in the polygons 
		and the segment $OH$ is bigger than $h_i$ for $P=P_1, P_2, P_2^*$. 

		(2) The distance between an edge and the diameter $AA'$  is larger than $h_i$ except the nearest and second nearest edges  for  either $P=P_1$ and  $g\ge3$,  or   $P=P_2$ and $g\ge2$,  or $P=P_1^*$ and $g\ge4$. 

		\label{prop_intersect_one_10}
\end{proposition}
\begin{proof}
(1) In Figure \ref{fig_vert_to_diameter_1_10}, $AA'$ is a diameter of the polygon
, $O$ is the center of either 4g-gon $P_1$, or (4g+2)-gon $P_2$,  or the center of the $2g+1$-gon of $P_2^*$ containing $A$. 
		$V$ is a vertex of the polygon 
		$P_1, P_2$ or the $2g+1$ polygon of $P_2^*$ containing $A$. 

		. Thus $|OV|=|OA| $. By Proposition \ref{prop_length_10} (3), $\cosh |OV|=\cot a\cot b$. 
 $\alpha=2ka$, $k\in \mathbb{N}$.  
 Distance between $AA'$ and $V$ is realized by $|VR|$. By  (\ref{for_tri_3}), \begin{eqnarray*}
		 \sinh |VR| &=& \sinh |OV| \sin \alpha \\ 
		 &=& \sqrt{\cot^2 a\cot^2 b- 1}\sin 2ka \\ &\ge&  \sqrt{\cot^2a\cot^2b - 1}\sin{ 2a}.
 \end{eqnarray*} 
 We use computer to compare $|VR|$ and $h_i$, finding that $|VR|\ge h_i$ if $ P = P_1, P_2$ 
 or the $2g+1$ polygon of $P_2^*$ containing $A$. Vertices of the other $2g+1$ polygon of $P_2^*$ do not meet $B(x,h_2)$ is a corollary of Proposition \ref{prop_intersect_one_10} (2). 
		\begin{figure}[htbp]
				\centering
				\includegraphics{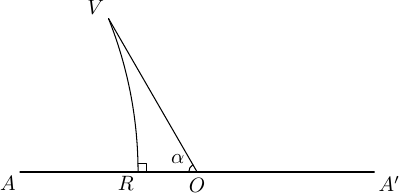}
				\caption{}
				\label{fig_vert_to_diameter_1_10}
		\end{figure}

(2) In Figure \ref{fig_vert_to_diameter_2_10}, $AA'$ is a diameter and $O$ is the center of $P_1, P_2$ or the center of the $2g+1$ polygon of $P_2^*$ containing $A$. $E$ is the mid-point of an edge of the polygon $P_1, P_2$ or the $2g+1$ polygon of $P_2^*$ containing $A$. $|OE|$ is the distance from the edge to $O$. By Proposition \ref{prop_length_10} (1), $\cosh |OE|= \cos b/\sin a$. $\beta=(2k+1)a$, $k\in \mathbb{N}$.  Distance between $AA'$ and the edge is realized by $|R_1R_2|$. By  (\ref{for_trirect_3}), 
		\begin{eqnarray*}
				\cosh |R_1R_2| &=& \cosh |OE| \sin \beta \\
				&=& \frac{\cos b \sin (2k+1)a}{\sin a}.
		\end{eqnarray*} Then we compare $|R_1R_2|$ and $h_i$ by computer. We get: $|R_1R_2| < h_i$ when $k\le 2$; $|R_1R_2| \ge h_i$ when $k \ge 3$, if $P=P_1$ and $g\ge 3$ or $P=P_2$ and $g\ge 2$ or $P=P_2^*$ and $g\ge 4$. 
		
		\begin{figure}[htbp]
				\centering
				\includegraphics{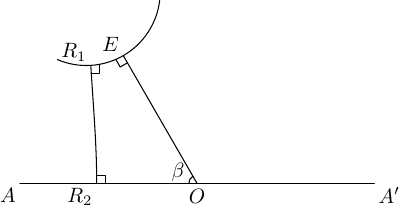}
				\caption{}
				\label{fig_vert_to_diameter_2_10}
		\end{figure}
\end{proof}

\begin{proposition}
		The upper half of a lift of $B(x,h_i)$ in $\mathbb{H}^2$ doesn't intersect any edges of the tessellation induced by the polygon 
		other than edges $AB_1, AB_2, B_1C_1$ in Figure \ref{fig_4g_plus_2_gon_univ_2}, if $P=P_1$ and $g\ge4$; if $P=P_2$ and $g\ge7$ or if $P=P_2^*$ and $g\ge3$. 

		\label{prop_intersect_10}
\end{proposition}

\begin{figure}[htbp]
		\centering
		\includegraphics{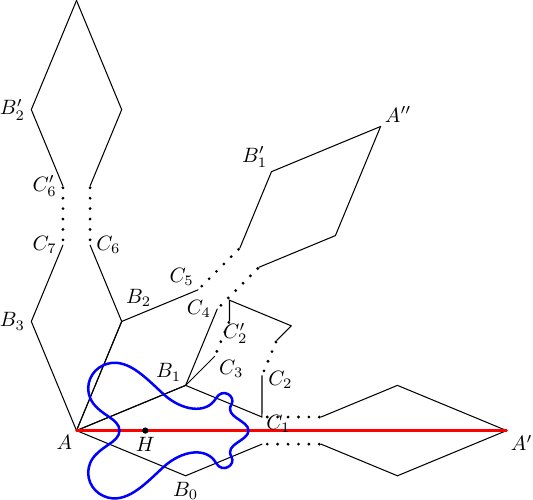}
		\caption{}
		\label{fig_4g_plus_2_gon_univ_2}
\end{figure}

\begin{figure}[htbp]
		\centering
		\includegraphics{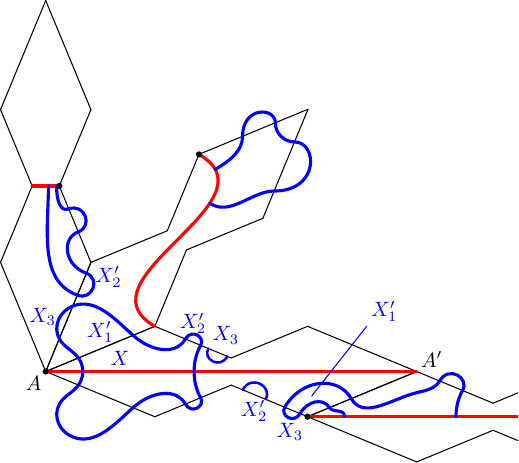}
		\caption{}
		\label{fig_4g_gon_univ_3}
\end{figure}

This proposition characterizes the shape of $B(x,h_i)$.

\begin{proof}
		By Proposition \ref{prop_intersect_one_10}, $B(x,h_i)$ may intersect $AB_1$ and $B_1C_1$ in their interior in all of the three models. 

				First, we calculate the distance between the segment $OH$ and $AB_j$ ($j=1,2,3$).
		\begin{lemma}
		The distance between $OH$ and $AB_j$ ($j=1,2$) is smaller than $h_i$, while the distance between $OH$ and $AB_3$ is bigger than $h_i$, if $P=P_1$ and $g\ge4$, if $P = P_2$ and $g\ge7$, if $P = P_2^*$ and $g\ge 3$. 
				\label{lem_abi}
		\end{lemma}

		\begin{proof}
				We calculate this result by Figure \ref{fig_abi}. In this figure, $AA'$ is a diameter and $AB_j$ is an edge. $x$ is a point on $AA'$ between $H$ and $O$. Thus $\inf d(x,A) = |AH|$ for $x$ between $H$ and $O$. $\angle A=(2j-1)b$. Then the infimum distance $Rx$ between $x$ and $AB_i$  is 
				\begin{eqnarray*}
						\inf_{x\in OH}\sinh |Rx| &=& \sin A \inf_{x\in OH}\sinh Ax \,\,\,\,(\text{by (\ref{for_tri_3})})\\ 
						&=&  (\sin (2j-1)b)\sinh |AH| \\
						&=&
				(\sin (2j-1)b )\cot b \sqrt{\frac{\cos^2 a - \sin^2 b}{\cos^2 a + \cos^2 b}} \,\,\,\,(\text{by Prop. 3(2)})
				\end{eqnarray*}
				by  (\ref{for_tri_3}).  
				Then we compare $\inf_{x\in OH}|Rx|$ and $h_i$ by computer programming. We get $\inf_{x\in OH}|Rx| \le h_i$ for $j = 1, 2 $ and $\inf_{x\in OH}|Rx| > h_i$ for $j=3$, if $P=P_1$ and $g\ge4$; if $P=P_2$ and $g\ge7$ or if $P=P_2^*$ and $g\ge3$. Therefore, the lemma is proved. 
\end{proof}

Lemma \ref{lem_abi} is equivalent to say that  $B(x,h_i)$ may intersect $AB_1$ or $AB_2$, but won't meet $AB_3$ or any other further edges whose one vertex is $A$.

		\begin{figure}[htbp]
				\centering
				\begin{minipage}[t]{0.45\textwidth}
\centering
				\includegraphics{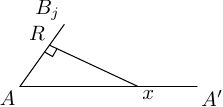}
				\caption{}
				\label{fig_abi}
				\end{minipage}
				\begin{minipage}[t]{0.45\textwidth}
				\centering
				\includegraphics{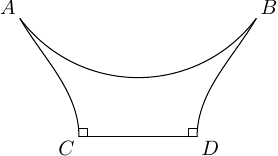}
				\caption{}
				\label{fig_trapezoid_2}
		\end{minipage}
		\end{figure}
		\begin{lemma}
				$h_i$ is smaller than the distance between any two edges of the polygon that does not intersect.

				\label{lem_dist_edges_10}
		\end{lemma}

		\begin{proof}
				It is sufficient to calculate the distance between two nearest non-intersecting edges in a polygon. 
				In Figure \ref{fig_trapezoid_2}, $A, B$ are vertices of the polygon, $\angle A =\angle B = 2b$. $AC$ and $BD$ are parts of edges of the polygon. $AB$ is an edge and thus $|AB|=2\arccosh \cos a/\sin b $ by Proposition \ref{prop_length_10}. We calculate $|CD|$ by  (\ref{for_birect_2}). Here \begin{eqnarray*}
						\cosh |CD| &=& \cosh |AB| \cdot \sin \angle A \sin \angle B  - \cos \angle A \cos \angle B \\
						&=& \left( \frac{2\cos^2 a}{\sin^2 b}-1 \right)\sin^2 2b - \cos^2 2b.
				\end{eqnarray*} We compare $|CD|$ with $h_i$, finding that $|CD| > h_i$ in all of the three models and obtain this conclusion. 

		\end{proof}

		By Lemma \ref{lem_dist_edges_10}, the $B(x,h_i)$ doesn't meet any edges in Figure \ref{fig_4g_plus_2_gon_univ_2} except $B_1C_3$, $C_1C_2$, $B_1C_4$, $B_2C_6$ and $AB_j$ ($j=1,2$). 
		
		\begin{lemma}
		$B(x,h_i)$ doesn't meet $B_1C_4$ and $B_2C_6$.
		\label{lem_x_b1c4}
		\end{lemma}

		\begin{proof} The geodesic segment connecting the middle point of $AB_0$ and $AB_1$ is perpendicularly bisected by the diameter $AA'$; the geodesic segment connecting the middle point of $B_1A$ and $B_1C_4$ is perpendicularly bisected by the diameter $B_1B_1'$. The length of each segment is $h_i$ by Propostion 3. These two geodesic segments form a new geodesic segment which is perpendicular to both $AA'$ and $B_1B_1'$. Therefore, the distance between $AA'$ and $B_1B_1'$ is $h_i$. 
Thus $B(x,h_i)$ doesn't meet $B_1C_4$ since segment $B_1C_4$ is farthur to $AA'$ than the line $B_1B_1'$. 

		Similarly, distance between the diameters $B_2B_2'$ and $AA''$ is not smaller than $h_i$. It implies that $B(x,h_i)$ doesn't meet $B_2C_6$. 

		We remind the reader that the two endpoints of the common perpendicular between $AA'$ and $B_1B_1'$ are $H$ and $H$'s deck transformation image on $B_1B_1'$. This fact is useful in the proof of Corollary \ref{cor_number}. 
		\end{proof}
		
		\begin{lemma}
		$B(x,h_i)$ doesn't meet $B_1C_3$ and $C_2C_1$. 
		\end{lemma}

\begin{proof}		First, we use Figure \ref{fig_trapezoid_model_1}, which is a part of Figure \ref{fig_4g_plus_2_gon_univ_2}, to calculate the distance between $B_1C_3$ and $AA'$ in Figure \ref{fig_4g_plus_2_gon_univ_2}. In Figure \ref{fig_trapezoid_model_1}, $AB_1$ is an edge of the polygon and therfore its length is $2\arccosh \cos a/\sin b$. $\angle A = b$ while $\angle B_1 = 4b$. Then the distance between $AA'$ and $B_1C_3$ (realized by $R_1R_2$ in the figure) is 
		\begin{eqnarray*}
			\cosh|R_1R_2| &=& \sin \angle A \sin \angle B_1 \cosh |AB_1|-\cos \angle A \cos \angle B_1 	\\
			&=& \sin b \sin 4b \left( 2\frac{\cos^2 a}{ \sin^2 b} -1 \right)-\cos b \cos 4b 	
		\end{eqnarray*} 
		 by  (\ref{for_birect_2}). By computer calculation, we know this distance is larger than $h_i$ in all of the three models. 

		\begin{figure}[htbp]
				\centering
				\begin{minipage}[t]{0.45\textwidth}
				\centering
				\includegraphics{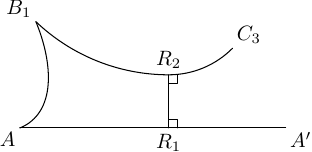}
				\caption{}
				\label{fig_trapezoid_model_1}
				\end{minipage}
				\begin{minipage}[t]{0.45\textwidth}
				\centering
				\includegraphics{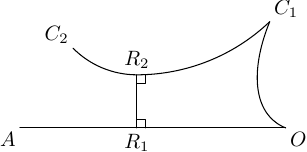}
				\caption{}
				\label{fig_trapezoid_model_2}
				\end{minipage}
		\end{figure}

		Then, we use Figure \ref{fig_trapezoid_model_2},   which is a part of Figure \ref{fig_4g_plus_2_gon_univ_2}, to calculate the distance between $C_1C_2$ and $AA'$ in Figure \ref{fig_4g_plus_2_gon_univ_2}.  In Figure \ref{fig_trapezoid_model_2}, Since $C_1$ is a vertex of the regular polygon containing $A$, $|OC_1|$ is $\arccosh \cot a \cot b$. $\angle O = 4a$ while $\angle C_1 = 3b$. Then the distance, realized by $R_1R_2$, between $AA'$ and $C_1C_2$ is 
		
		\begin{eqnarray*}
		\cosh |R_1R_2| &=& \sin \angle O \sin \angle C_1 \cosh |OC_1|-\cos \angle A \cos \angle B_1 \\
		&=& \cot a\cot b \sin 4a \sin 3b - \cos 4a \cos 3b
		\end{eqnarray*}
		by  (\ref{for_birect_2}). We compare this distance with $h_i$, finding that $|R_1R_2|$ is larger than $h_i$ in all of the three models. 
\end{proof}\end{proof}

Now we have proved Proposition \ref{prop_intersect_10}, that is $B(x, h_i)$ in the universal cover $\mathbb{H}^2$ is shown as in Figure \ref{fig_4g_plus_2_gon_univ_2}.

By projecting $B(x, h_i)$ from the universal cover to $\Sigma_g^i$, we will get a picture of $B(x, h_i)\subset P_i (P_i^*)\subset \Sigma_g^i$ shown as in Figure \ref{fig_4g_plus_2_gon_6}

\subsection{Back to $P$}

\begin{figure}[htbp]
		\centering
		\includegraphics{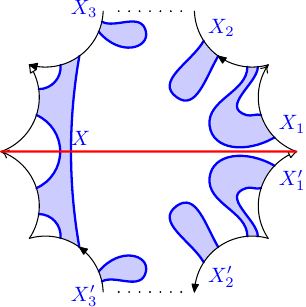}
		\caption{}
		\label{fig_4g_plus_2_gon_6}
\end{figure}

Now it is sufficient to prove that any two components in Figure \ref{fig_4g_plus_2_gon_6} does not intersect. 

Two components in the polygon do not intersect each other if and only if the corresponding disks (lifts of $B(x,h_i)$) in $\mathbb{H}^2$ does not intersect.

\begin{lemma}
The interior of $X_j$ does not meet the diameter $AA'$ for $j=1,2,3$
\label{lem_xi_dia}
\end{lemma}

\begin{proof} $X_1$'s interior intersects $AA'$ in Figure \ref{fig_4g_plus_2_gon_6} if and only if in Figure \ref{fig_4g_plus_2_gon_univ_2}, $B(x,h_i)$'s interior intersects the diameter $B_1B_1'$. But $d(AA',B_1B_1')=h_i$. Thus $X_1$'s interior does not intersect $AA'$. 

To prove $X_2$ in Figure \ref{fig_4g_plus_2_gon_6} not meeting $AA'$, it is sufficient to prove in Figure \ref{fig_4g_plus_2_gon_univ_2}, the distance between $AA'$ and $C_2C_2'$ is larger than $h_i$. 

In Figure \ref{fig_x2a_10}, $d(AA',C_2C_2')$ is realized by segment $R_5R_6$. By using  (\ref{for_trirect_3}) in trirectangle with right angles $R_5, R_1, R_2$ and trirectangle with right angles $R_6, R_3, R_4$, we have $|R_5R_6|>|R_3R_4|+|R_1R_2|$. We calcuate $|R_3R_4|$ and $|R_1R_2|$ in the quadrilateral $R_3R_4B_1A$ and $R_1R_2C_1C_2$ respectively by  (\ref{for_birect_2}).
$\cosh |R_3R_4| = \cosh |AB_1| \sin \angle B_1AR_3\sin \angle AB_1R_4 - \cos \angle B_1AR_3\cos \angle AB_1R_4$ and $\cosh |R_1R_2| = \cosh |C_1C_2|\sin \angle R_1C_2C_1 \sin \angle C_2C_1R_2 - \cos \angle R_1C_2C_1 \cos \angle C_2C_1R_2$. 
Here $\angle B_1AR_3= \angle R_1C_2C_1 = b$, $\angle AB_1R_4= \angle C_2C_1R_2 = 2b$, $|AB_1|=|C_1C_2|=2\arccosh \cos b/\sin a $.
Therefore, 
\[
		\cosh |R_1R_2| = \cosh |R_3R_4| = \left( 2\frac{\cos^2 a}{\sin^2 b} -1 \right)\sin 2b \sin b - \cos 2b \cos b .
\]
We calculate $|R_1R_2|$ and $|R_3R_4|$ by computer, then compare $|R_1R_2| + |R_3R_4|$ with $h_i$ and get the following conclusion: 
 $|R_3R_4|+|R_1R_2|>h_i$, and $X_2\cap AA'= \emptyset$ in Figure \ref{fig_4g_plus_2_gon_6} for all of the three models. 

\begin{figure}[htbp]
		\centering
		\includegraphics{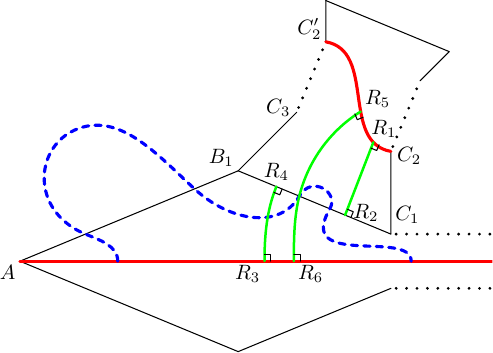}
		\caption{}
		\label{fig_x2a_10}
\end{figure}

$X_3$ in Figure \ref{fig_4g_plus_2_gon_6} does not intersect $AA'$ if and only if $B(x,h_i)$ in Figure \ref{fig_4g_plus_2_gon_univ_2} does not intersect the diameter $C_6C_7$. $B(x,h_i)$ in Figure \ref{fig_4g_plus_2_gon_univ_2} does not intersect the diameter $C_6C_7$, if $d(AA',C_6C_7)>h_i$. 

We use Figure \ref{fig_x3a_10} to calculate $d(AA',C_6C_7)$. Figure \ref{fig_x3a_10} is a part of Figure \ref{fig_4g_plus_2_gon_univ_2}. $O'$ is the center of the regular polygon it is in. In the quadrilateral $AO'R_2R_1$, $\angle A= 4b ,\angle O'= 4a$, $|O'A|=\arccosh \cot a \cot b$ by Proposition \ref{prop_length_10}. Then by  (\ref{for_birect_2}), 
\begin{align}
		 \cosh |R_1R_2| &= \cosh |AO'| \sin \angle A \sin \O' - \cos A \cos O' \nonumber\\
		 &= \cot a \cot b \sin 4a \sin 4b - \cos 4a \cos 4b. \label{for_r1r2}
\end{align}
 By the help of the computer, we compare $|R_1R_2|$ and $h_i$, finding that $|R_1R_2|$ is larger than $h_i$ for all of the three models. 

\begin{figure}[htbp]
		\centering
		\includegraphics{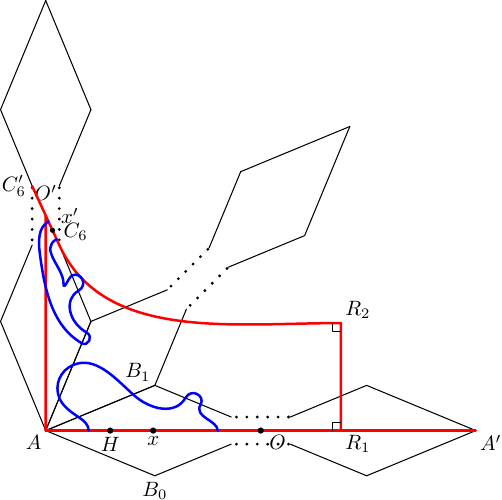}
		\caption{}
		\label{fig_x3a_10}
\end{figure}
\end{proof}

\begin{figure}[htbp]
		\centering
		\includegraphics{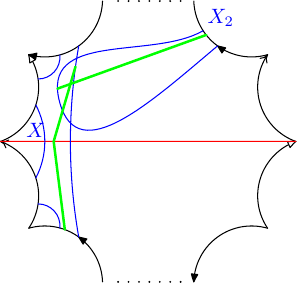}
		\caption{}
		\label{fig_xxi_10}
\end{figure}

\begin{lemma}
In Figure \ref{fig_4g_plus_2_gon_6}, 
 $X_1, X_2, X_3$ does not intersect $X$. 	
 \label{lem_xxi}
\end{lemma}

\begin{proof}

(1) The distance between the polygon's center and an edge is $\cos b/\sin a$ by Proposition \ref{prop_length_10}. Then we compare it with $h_i$ by computer and the conclusion follows. For $g\ge 4$, distance between the polygon's center and edges are bigger than $h_i$.	If $X$ intersects $X_1$ or $X_2$, there are two radii of $B(x,h_i)$ that intersect. It imples that the distance between two opposite edges is smaller or equal to $2h_i$, which is impossible. (See Figure \ref{fig_xxi_10}. )

(2) $X\cap X_3 = \emptyset$ in Figure \ref{fig_4g_plus_2_gon_6} is equivalent to in Figure \ref{fig_4g_plus_2_gon_univ_2}, the distance between $x$ on $AA'$ and its deck transformation image on $C_6C_7$ is larger than $2h_i$. (See Figure \ref{fig_x3a_10}) Now we prove this distance is larger than $2h_i$ by the formulae of the quadrilateral: 

$x$ is a point moving on the diameter $AA'$ between $H$ and $O$. $x'$ is its deck transformation image. $x'$ moves as $x$ moves. We calculate $\inf_{x\in HO} d(x,x')$ and compare it with $2h_i$. 

In Figure \ref{fig_x3a_10}, we have obtained $|R_1R_2|$ in (\ref{for_r1r2}). Then in the quadrilateral $R_1R_2x'x$, by  (\ref{for_birect_1}), 
\[
		\cosh |xx'| = \cosh |R_1R_2| \cosh |R_1x| \cosh |R_2x'| - \sinh |R_1x| \sinh |R_2x'|. 
\] 

Assume $|Ax| = t$, then $|R_1x| = |AR_1| - t$ and $|R_2x'| = |R_2C_6|+t$. Here $|C_6x'| = |Ax| =t$.  $|AH|\le t \le |AO|$. Here $H$ is the intersecting point of the diameter and the geodesic connecting the mid-point of $AB_0$ and $AB_1$. 

Then we denote $\cosh |xx'|$ to be $f(t)$. 
\[
		f(t) = \cosh |R_1R_2| \cosh (|AR_1|-t) \cosh (|C_6R_2|+t) -\sinh (|AR_1|-t) \sinh (|C_6R_2|+t).
\] 
Thus 
\begin{eqnarray*}
		f'(t) &=& \cosh |R_1R_2| (-\sinh (|AR_1|-t) \cosh (|C_6R_2|+t) + \cosh (|AR_1|-t) \sinh (|C_6R_2|+t) ) - \\ 
		& &(-\cosh (|AR_1|-t) \sinh (|C_6R_2|+t) + \sinh (|AR_1|-t) \cosh (|C_6R_2|+t)) \\
		&=& (\cosh |R_1R_2| -1) \sinh (2t - |AR_1| + |C_6R_2|). 
\end{eqnarray*}

Therefore, the minimum of $f(t)$ is obtained when $t = (|AR_1| - |C_6R_2|)/2$. At this point, $|R_1x| = |R_2x'| = (|AR_1| + |C_6R_2|)/2$ and thus quadrilateral $xx'R_2R_1$ is a quadrilateral with $\angle R_1 = \angle R_2 =\pi/2$, $\angle x = \angle x'$. Moreover, $|Ax| + |O'x'| = |AO|$. ($|O'x'| = |O'C_6| - |C_6x'| = |AO|- |Ax|$. )

To prove $|xx'|>2h_i$, we construct a quadrilateral $R_1R_2S_2S_1$.(See Figure \ref{fig_compare_xx_yy}) In the quadrilateral $R_1R_2O'A$, $\angle A = 4b$ and $\angle O' = 4a$. Without loss of generality, we assume $b\ge a$. $S_1\in AR_1$, $S_2 \in O'R_2$, satisfying $\angle S_1S_2R_2 = \angle S_2S_1R_1 = 4b$. (We remark that when $a=b$, $S_1=A$ and $S_2=O'$). Such $S_1$ and $S_2$ exist on the segment $R_1A$ and $R_2O'$ respectively instead of on the extended lines of $R_1A$ and $R_2O'$ because otherwise we'll get a quadrilateral with sum of interior angles bigger than $2\pi$ or a triangle with sum of interior angles bigger than $\pi$, which is impossible.  

\begin{figure}[htbp]
		\centering
		\includegraphics{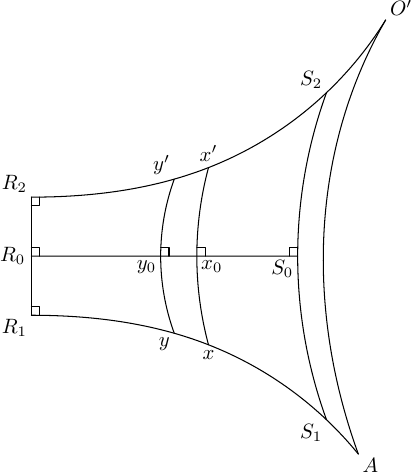}
		\caption{}
		\label{fig_compare_xx_yy}
\end{figure}

Then we pick up two points $y$, $y'$ on $R_1S_1$ and $R_2S_2$ respectively, satisfying $|S_1y| = |S_2y'| = |AO|/2$. Then $|R_1x|\ge|R_1y|$ and $|R_2x'|\ge|R_2y'|$ by the following argument: We have $|R_1A|\ge |R_1S_1|$, $|R_2O'|\ge |R_2S_2|$ and $|Ax|+|O'x'| =|AO|$. Without loss of generality, we assume $|Ax|\le|AO|/2\le|O'x'|$. Then $|R_1y| = |R_1S_1| - |S_1y| = |R_1S_1| - |AO|/2\le |R_1A| - |Ax| = |R_1x|$. Then $|R_2y'|\le |R_2x'|$ follows from $|R_1x| = |R_2x'|$ and $|R_1y| = |R_2y'|$. This result means $|xx'|\ge|yy'|$ by  (\ref{for_birect_2}). 

To prove $|xx'|\ge 2h_i$, it is sufficient to prove $|yy'|\ge 2h_i$. Now we calculate $|yy'|$ and compare it with $2h_i$. 
By the symmetry of the quadrilaterals $R_1R_2y'y$, $R_1R_2x'x$ and $R_1R_2S_2S_1$, the segment $R_0S_0$ divides each of these quadrilaterals into two equal trirectangles, namely $R_0R_1yy_0$ and $R_0R_2y'y_0$; $R_0R_1xx_0$ and $R_0R_2x'x_0$; $R_0R_1S_1S_0$ and $R_0R_2S_2S_0$. Here $R_0$, $y_0$, $x_0$ and $S_0$ are the middle points of $R_1R_2$, $yy'$, $xx'$ and $S_1S_2$ respectively.  Thus $|yy'|\ge 2h_i$ is equivalent to $|yy_0|\ge h_i$. Now we calculate $|yy_0|$.

By our construction, in the trirectangle $R_1R_0S_0S_1$, $\angle R_0 = \angle R_1 = \angle S_0 = \pi/2$, $\angle S_1 = 4b$. $|R_1R_2|$ is obtained in (\ref{for_r1r2}). Then we get $\coth \frac{|R_1R_2|}{2}$:  
\begin{align*}
		\cosh \frac{|R_1R_2|}{2} &= \sqrt{\frac{\cosh |R_1R_2|+1}{2}} \\
		&= \sqrt{\frac{\cot a \cot b \sin 4a \sin 4b - \cos 4a \cos 4b +1}{2}},
\end{align*}

\begin{align*}
		\sinh \frac{|R_1R_2|}{2} &= \sqrt{\frac{\cosh |R_1R_2|-1}{2}} \\
		&= \sqrt{\frac{\cot a \cot b \sin 4a \sin 4b - \cos 4a \cos 4b -1}{2}} ,
\end{align*}

\begin{align*}
		\coth \frac{|R_1R_2|}{2} &= \cosh \frac{|R_1R_2|}{2}/\sinh \frac{|R_1R_2|}{2} \nonumber \\
		&= \sqrt{\frac{\cot a \cot b \sin 4a \sin 4b - \cos 4a \cos 4b +1}{\cot a \cot b \sin 4a \sin 4b - \cos 4a \cos 4b -1}}. 
\end{align*}

Then by  (\ref{for_trirect_6}), we obtain $|R_1S_1|$: 
\begin{align*}
		\sinh |R_1S_1| &= \coth \frac{|R_1R_2|}{2} \cot \angle S_0S_1R_1 \\
		&=  \coth \frac{|R_1R_2|}{2} \cot 4b
\end{align*}

Then we obtain $|R_1y|$:
\begin{align*}
		\cosh |R_1y| &= \cosh (|R_1S_1| - |S_1y|) \\
		&= \cosh (|R_1S_1| - |AO|/2) \\
		&= \cosh |R_1S_1| \cosh \frac{|AO|}{2} - \sinh|R_1S_1| \sinh\frac{|AO|}{2}
\end{align*}
Here $|AO| = \arccosh (\cot a \cot b)$. 

Finally, we get $|yy_0|$ by  (\ref{for_trirect_5}):
\begin{align*}
		\sinh |yy_0| = \sinh |R_0R_1| \cosh |R_1y|
\end{align*}.

By computer programming, we compare $yy_0$ and $h_i$, finding that $yy_0 > h_i$  and therefore $X\cap X_3 = \emptyset$ in Figure \ref{fig_x3a_10}. 

We remark that the proof above also proves $X_1\cap X_2 = \emptyset$, see Figure \ref{fig_x3a_10}. 

\end{proof}

\begin{lemma}
In Figure \ref{fig_4g_plus_2_gon_6}, 
		$X_1, X_2, X_3$ do not intersect each other when $g\ge 4$. 
		\label{lem_xixj}
\end{lemma}

\begin{proof}
		The proof of $X\cap X_3 = \emptyset$ in Lemma \ref{lem_xxi} also proves $X_1\cap X_2 = \emptyset$. 
Below we  prove  (1) $X_2\cap X_3=\emptyset$ (2) $X_1\cap X_3=\emptyset$. 

In $P_1$ model, we pick the diameter perpendicular to $AA'$ (denoted $BB'$). Now we prove that $X_1, X_2, X_3$ cannot meet $BB'$, and therefore $X_1\cap X_3=\emptyset$ and $X_2\cap X_3=\emptyset$. (See Figure \ref{fig_4g_gon_10}. )

\begin{figure}[htbp]
		\centering
		\begin{minipage}[t]{0.45\textwidth}
				\centering
				\includegraphics[scale=0.9]{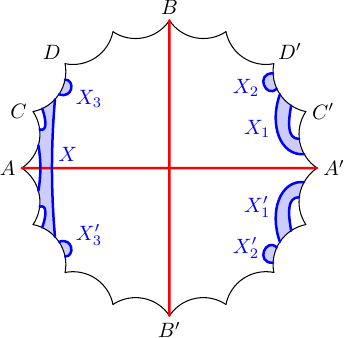}
				\caption{}
				\label{fig_4g_gon_10}
		\end{minipage}
		\begin{minipage}[t]{0.45\textwidth}
				\centering
				\includegraphics[scale=0.9]{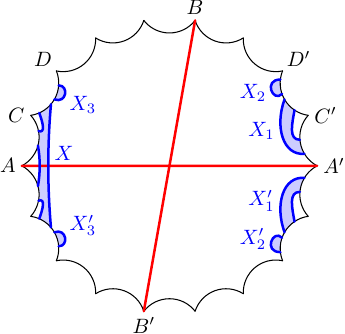}
				\caption{}
				\label{fig_4g_plus_2_gon_11}
		\end{minipage}
\end{figure}
 The proof of $X_i\cap BB' = \emptyset$ are exactly the same for $i = 1,2,3$.
Without loss of generality, we prove that $X_3 \cap BB' = \emptyset$. , In Figure \ref{fig_4g_gon_10}, we denote the edge that meets $X_3$ to be $CD$. Since $g\ge 4$, the edge $CD$ is not the nearest  and the second nearest edge to the diameter $BB'$, by Proposition \ref{prop_intersect_one_10} (2), $d(BB',CD)>h_i$. On the other hand, the center $x'$ of the $B(x',h_i)$ containing $X_3$ is outside the polygon in Figure \ref{fig_4g_gon_10}. Therefore, $\forall x_3 \in X_3$, $d(x_3,CD)\le h_i$. Thus $X_3\cap BB' = \emptyset$. 

In $P_2$ model,  let $BB'$ be one of the diameters whose angle with $AA'$ is the biggest, see Figure \ref{fig_4g_plus_2_gon_11}. 
Since $g\ge 4$ and $P_2$ is 4g+2 gon, we still can apply Proposition \ref{prop_intersect_one_10} (2) to prove $X_i\cap BB' = \emptyset$ for $i = 1,2,3$ exactly as $P_1$ case, and then $X_1 \cap X_3 = \emptyset$, $X_2 \cap X_3 = \emptyset$. 

\begin{figure}[htbp]
		\centering
		\includegraphics{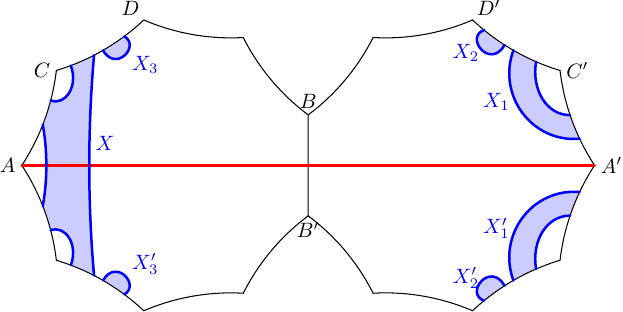}
		\caption{}
		\label{fig_4g_plus_2_gon_5_11}
\end{figure}

In $P_2^*$ model, we let $BB'$ be the edge seperate the two $2g+1$ polygons and the edge that meets $X_3$ be $CD$ (See Figure \ref{fig_4g_plus_2_gon_5_11}) Since $g\ge 4$, $2g+1\ge 9$, the edge $CD$ and the edge $BB'$ are disjoint. Then by Lemma \ref{lem_dist_edges_10}, $d(BB',CD)>h_i$. On the other hand, the center $x'$ of the $B(x',h_i)$ that corresponds to $X_3$ is outside the polygon in Figure \ref{fig_4g_plus_2_gon_5_11}. Therefore, $\forall x_3 \in X_3$, $d(x_3,CD)\le h_i$. Thus $X_3\cap BB' = \emptyset$. 

By exactly the same proof, we have $X_i\cap BB' = \emptyset$ for $i = 1,2$. So that $X_1 \cap X_3 = \emptyset$ and $X_2 \cap X_3 = \emptyset. $

We have proved the theorem.

\end{proof}

		\setcounter{cor}{0}
		Now we begin to prove Corollary \ref{cor_number}. 
\begin{cor}
In $\Sigma_g^1$, there are $2g$ closed geodesics having length  $\sys(\Sigma_g^1)$
and in $\Sigma_g^2$, there are $2g+1$ closed geodesics having length  $\sys(\Sigma_g^2)$.
		\label{cor_number}
\end{cor}

\begin{proof}
		By the proof of Lemma \ref{lem_xi_dia}, Lemma \ref{lem_xxi} and Lemma \ref{lem_xixj}, each pair of components in Figure \ref{fig_4g_plus_2_gon_6} does not intersect except $X_1$ and $X_1'$. $X_1$ and $X_1'$ may be tangent to each other. For a fixed $x$ (center of the ball $B(x,h_i)$) on $OH$, the number of points in $\partial X_1 \cap \partial X_1'$ is $0$ or $1$ by the convexity of balls in hyperbolic plane. The only thing to show is $\partial X_1 \cap \partial X_1' \ne \emptyset$ if and only if $x$ is the point $H$. 

		This fact is straight forward by the proof of Lemma \ref{lem_x_b1c4}. 
		In that proof, the common perpendicular between $AA'$ and $B_1B_1'$ in Figure \ref{fig_4g_plus_2_gon_univ_2} is the segment connecting $H$ and $H$'s deck transformation image on $B_1B_1'$. Besides, $d(AA', B_1B_1') = h_i$. Therefore $d(x,B_1B_1')\ge h_i$, $\forall x\in OH$. The equality holds if and only if $x=H$. 
		It proves that in Figure \ref{fig_4g_plus_2_gon_6}, $X_1\cap AA' = \{H'\}$ and $X_1'\cap AA' = \{H'\}$ if and only if $x = H$. (Here $H'$ is the intersecting point of $E'D'$ and $AA'$ in Figure \ref{fig_4g_plus_2_gon_10}. ) 

		Then it proves that $\gamma_i$ ($DEE'D'$ in Figure \ref{fig_4g_plus_2_gon_10}) is the unique systole that intersects $OH$ in all of the three models. By Claim \ref{claim_final}, $\gamma_i$ is the unique systole that intersects $OA$ in $P_1$ and $P_2$. 
		By the invariance of $\gamma_i$ under the $\pi$-rotation of $P_1$ and $P_2$, $\gamma_i$ is the unique systole that intersects $AA'$ in $P_1$ and $P_2$. 
		This is equivalent to that given a diameter of $P_1$ or $P_2$, there is a unique systole intersects the diameter. Therefore by counting the number of diameters of $P_1$ and $P_2$, the Corollary holds. 

\end{proof}

\section{Appendix}
In this Section, we give the source code and figure for the comparison between $|CD|$ and $h_i$ in Lemma \ref{lem_dist_edges_10}. The code is written in MATLAB. The codes for other comparisons are similar. 

\begin{lstlisting}

%P_1 model
a = [eps:eps:pi/8];
b = a;
%P_2 model
%a = [eps:eps:pi/10];
%b = 2*a;
%P_2^* model
%a = [eps:eps:pi/5];
%b = 0.5*a;

chdst = (2*(cos(a)./sin(b)).^2-1).*(sin(2*b)).^2-(cos(2*b)).^2;

chh = 1 + cos(2*a) + cos(2*b);
%shh = sqrt(chh.^2-1);
f = chdst - chh;
plot(a,f);
\end{lstlisting}

Figure \ref{fig_dist_between_two_edges} shows the result of the comparison in $P_1$ model. In Figure \ref{fig_dist_between_two_edges}, the horizontal axis is the variable $a$, while the vertical axis is $\cosh |CD| - \cosh h_i$ (depends on $a$). 

\begin{figure}[htbp]
		\centering
		\includegraphics[scale=0.6]{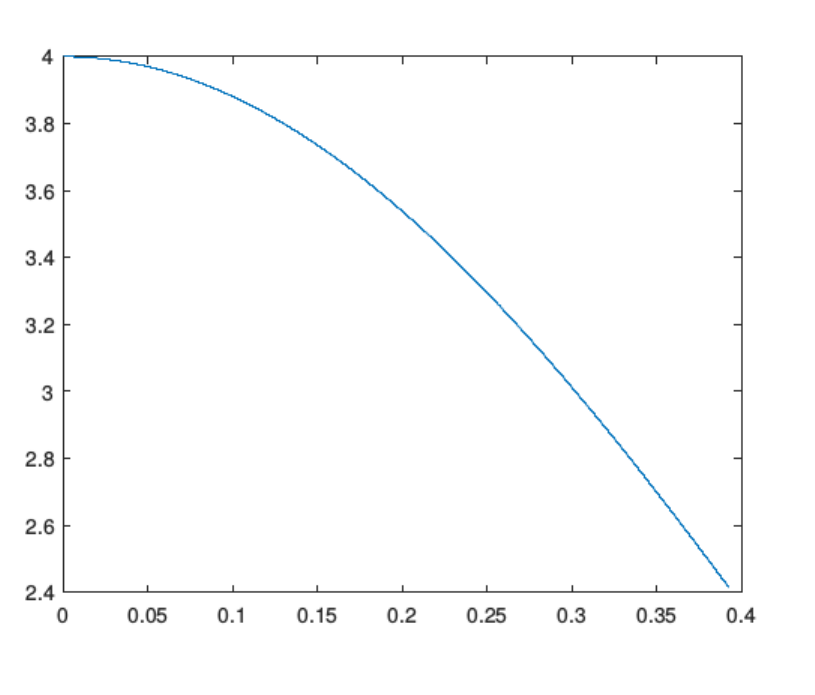}
		\caption{}
		\label{fig_dist_between_two_edges}
\end{figure}

\bibliographystyle{alpha}

\begin{thebibliography}{Akr93}


\bibitem[Bav92]{bavard1992systole}
Christophe Bavard.
\newblock La systole des surfaces hyperelliptiques.
\newblock {\em Prepubl. Ec. Norm. Sup. Lyon}, 71, 1992.

\bibitem[BS94]{buser1994period}
Peter Buser and Peter Sarnak.
\newblock On the period matrix of a riemann surface of large genus (with an
  appendix by jh conway and nja sloane).
\newblock {\em Inventiones mathematicae}, 117(1):27--56, 1994.

\bibitem[Bus10]{buser2010geometry}
Peter Buser.
\newblock {\em Geometry and spectra of compact Riemann surfaces}.
\newblock Springer Science \& Business Media, 2010.



\bibitem[GWWZ15]{guo2015embedding}
Yu~Guo, Chao Wang, Shicheng Wang, and Yimu Zhang.
\newblock Embedding periodic maps on surfaces into those on $s^3$.
\newblock {\em Chinese Annals of Mathematics, Series B}, 36(2):161–180, March
  2015.



\bibitem[Hu]{Hu} A. Hurwitz, {\it \"Uber algebraische Gebilde mit eindeutigen
Transformationen in sich}, Math. Ann. 41 (1893), 403-442


\bibitem[Jen84]{jenni1984ersten}
Felix Jenni.
\newblock {\"U}ber den ersten eigenwert des laplace-operators auf
  ausgew{\"a}hlten beispielen kompakter riemannscher fl{\"a}chen.
\newblock {\em Commentarii Mathematici Helvetici}, 59(1):193--203, 1984.

\bibitem[KSV07]{katz2007logarithmic}
Mikhail~G Katz, Mary Schaps, Uzi Vishne,
\newblock Logarithmic growth of systole of arithmetic riemann surfaces along
  congruence subgroups.
\newblock {\em Journal of Differential Geometry}, 76(3):399--422, 2007.

\bibitem[Kul97]{kulkarni1997riemann}
Ravi~S Kulkarni.
\newblock Riemann surfaces admitting large automorphism groups.
\newblock {\em Extremal Riemann surfaces (San Francisco, CA, 1995)},
  201:63--79, 1997.

\bibitem[Par14]{parlier2014simple}
Hugo Parlier.
\newblock Simple closed geodesics and the study of teichm{\"u}ller spaces.
\newblock {\em Handbook of Teichm{\"u}ller Theory, Volume IV}, pages 113--134,
  2014.

\bibitem[PAR06]{petersen2006riemannian}
Peter Petersen, S~Axler, and KA~Ribet.
\newblock {\em Riemannian geometry}, volume 171.
\newblock Springer, 2006.

\bibitem[Pet18]{petri2018hyperbolic}
Bram Petri.
\newblock Hyperbolic surfaces with long systoles that form a pants
  decomposition.
\newblock {\em Proceedings of the American Mathematical Society},
  146(3):1069--1081, 2018.

\bibitem[PW15]{petri2015graphs}
Bram Petri and Alexander Walker.
\newblock Graphs of large girth and surfaces of large systole.
\newblock {\em arXiv preprint arXiv:1512.06839}, 2015.


\bibitem[Sch93]{schmutz1993reimann}
P~Schmutz.
\newblock Reimann surfaces with shortest geodesic of maximal length.
\newblock {\em Geometric \& Functional Analysis GAFA}, 3(6):564--631, 1993.


\bibitem[Wan91]{wang1991maximum}
Shicheng Wang.
\newblock Maximum orders of periodic maps on closed surfaces.
\newblock {\em Topology and its Applications}, 41(3):255--262, 1991.

\bibitem[Wim95]{wiman1895ueber}
A. Wiman, {\it Uber die hyperelliptischen Kurven und diejenigen vom Geschlecht p=3, welche eindeutige Transformationen in sich zulassen,} Bihang Till. Kongl. Svenska Vetenskaps-Akademiens Handlingar 21 (1) 1895.

\end{thebibliography}

\end{document}